\documentclass[11pt]{amsart}

\usepackage{mathtools}
\mathtoolsset{showonlyrefs}

\usepackage[margin=1.2in]{geometry}
\def\eps{\epsilon}
\def\lam{\lambda}

\newtheorem*{theorem*}{Theorem}
\newtheorem{theorem}{Theorem}
\newtheorem{lemma}{Lemma}
\newtheorem{cor}{Corollary}

\newtheorem{prop}{Proposition}
\newtheorem*{prop*}{Proposition}
\newtheorem{conj}{Conjecture}

\newcommand*{\occ}{\frac{1}{n}\overline{\alpha}}
\DeclarePairedDelimiter{\abs}{\vert}{\vert}
\DeclareMathOperator{\E}{\mathbb{E}}
\DeclareMathOperator{\var}{Var}

\begin{document}
\title[On the average size of independent sets]{On the average size of independent sets \\ in triangle-free graphs}
\author{Ewan Davies}
\author{Matthew Jenssen}
\author{Will Perkins}

\author{Barnaby Roberts}
\address{London School of Economics}
\email{\{e.s.davies,m.o.jenssen,b.j.roberts\}@lse.ac.uk }
\address{University of Birmingham}
\email{math@willperkins.org}
\date{\today}

\begin{abstract}
We prove an asymptotically tight lower bound on the average size of independent sets in a triangle-free graph on $n$ vertices with maximum degree $d$, showing that an independent set drawn uniformly at random from such a graph has expected size at least $(1+o_d(1)) \frac{\log d}{d}n$.  This gives an alternative proof of Shearer's upper bound on the Ramsey number $R(3,k)$.  We then prove that the total number of independent sets in a triangle-free graph with maximum degree $d$ is at least  $\exp \left[\left(\frac{1}{2}+o_d(1) \right) \frac{\log^2 d}{d}n \right]$.  The constant $1/2$ in the exponent is best possible. In both cases,  tightness is exhibited by a random $d$-regular graph.

Both results come from considering the hard-core model from statistical physics:  a random independent set $I$ drawn from a graph with probability proportional to $\lambda^{|I|}$, for a fugacity parameter $\lambda>0$. We prove a general lower bound on the occupancy fraction (normalized expected size of the random independent set) of the hard-core model on triangle-free graphs of maximum degree $d$. The bound is asymptotically tight in $d$ for all $\lambda =O_d(1)$. 

We conclude by stating several conjectures on the relationship between the average and maximum size of an independent set in a triangle-free graph and give some consequences of these conjectures in Ramsey theory.  
\end{abstract}

\maketitle

\section{Introduction}

\subsection{Independent sets in triangle-free graphs}

Ajtai, Koml{\'o}s, and Szemer{\'e}di~\cite{ajtai1980note} proved that any triangle-free graph $G$ on $n$ vertices with average degree $d$ has an independent set of size at least $c \frac{\log d}{d} n$ for a small constant $c$.  Shearer~\cite{shearer1983note} later improved the constant to $1$, asymptotically as $d \to \infty$, showing that such a graph has an independent set of size at least $f(d) \cdot n$ where $f(d) = \frac{d\log d - d +1}{(d-1)^2} = (1+o_d(1)) \frac{\log d}{d} $. Here, and in what follows, logarithms will always be to base $e$. We use standard asymptotic notation, with subscripts indicating which parameters the implied functions depend on. For example we write $o_d(1)$ for a quantity that tends to zero as $d$ tends to infinity.

The off-diagonal Ramsey number $R(3,k)$ is the least integer $n$ such that any graph on $n$ vertices contains either a triangle or an independent set of size $k$.  The above result of Ajtai, Koml{\'o}s, and Szemer{\'e}di and a result of Kim~\cite{kim1995ramsey} show that $R(3,k) = \Theta(k^2/\log k)$.  Shearer's result gives the current best upper bound, showing that $R(3,k)\le (1+o(1))k^2/\log k$.  Independent work of Bohman and Keevash~\cite{bohman2013dynamic} and Fiz Pontiveros, Griffiths, and Morris~\cite{pontiveros2013triangle}  shows that $R(3,k) \ge (1/4+o(1))k^2/\log k$.  Reducing the factor $4$ gap between these bounds is a major open problem in Ramsey theory.

We prove a lower bound on the \emph{average size} of an independent set in a triangle-free graph of maximum degree $d$, matching the asymptotic form of Shearer's result, and in turn giving an alternative proof of the above upper bound on $R(3,k)$.

\begin{theorem}\label{thm:avgShort}
Let $G$ be a triangle-free graph on $n$ vertices with maximum degree $d$. Let $\mathcal I(G)$ be the set of all independent sets of $G$. 
Then 
\[ \frac{1}{|\mathcal I(G)|} \sum_{I \in \mathcal I(G)} |I| \ge (1+o_d(1))\frac{\log d}{d} n. \]
 Moreover, the constant `1' is best possible.
\end{theorem}

This result is weaker than Shearer's \cite{shearer1983note} in that instead of average degree $d$ we require maximum degree $d$. 
Our result is stronger in that we show that the \emph{average size} of an independent set from such a graph is of size at least $ (1+o_d(1)) \frac{\log d}{d} n $, while Shearer shows the largest independent set is of at least this size (by analyzing a randomized greedy algorithm).

The proof of Theorem~\ref{thm:avgShort} is an adaptation of the authors' method used in~\cite{davies2015independent} to give a strengthening of results of Kahn~\cite{kahn2001entropy} and Zhao~\cite{zhao2010number} that a $d$-regular graph has at most as many independent sets as a disjoint union of $K_{d,d}$'s on the same number of vertices. 
We make this explicit in Section~\ref{sec:proof}, giving a unified method for proving these classical results of Shearer and Kahn. After sharing a draft of this paper with colleagues, we discovered that James Shearer also knew the proof of the lower bound in Theorem~\ref{thm:avgShort} and presented a sketch of it at the SIAM Conference on Discrete Mathematics in 1998, but never published it \cite{ShearerEmail}.

To see that Theorem~\ref{thm:avgShort} directly implies the upper bound $(1+o(1))k^2/\log k$ on $R(3,k)$, 
suppose that $G$ is triangle free with no independent set of size $k$. 
Then $G$ must have maximum degree less than $k$. Applying Theorem~\ref{thm:avgShort} we see the independence number is at least $(1+o_k(1))\frac{\log k}{k} n$ but less than $k$, and so $n<(1+o_k(1))\frac{k^2}{\log k}$ as required. Of course this reasoning simply uses the average size of an independent set as a lower bound for the maximum size.  In Section~\ref{sec:conjectures} we consider whether the discrepancy between the maximum and average size can be exploited to improve the upper bound on $R(3,k)$. 

After bounding the average size of an independent set from $G$, we next give lower bounds on the total number of independent sets in $G$, $|\mathcal I(G)|$.  

\begin{theorem}\label{thm:countShort}
Let $G$ be a triangle-free graph on $n$ vertices with maximum degree $d$. Then 
\[ |\mathcal I(G)| \ge e^{\left(\frac{1}{2}+ o_d(1)\right) \frac{\log^2 d}{d}  n }.\]
Moreover, the constant $1/2$ in the exponent is best possible.
\end{theorem}

In comparison, improving on previous results of Cooper and Mubayi~\cite{cooper2014counting}, Cooper, Dutta, and Mubayi \cite{cooper2014hypergraphs} proved that any triangle-free graph of average degree $d$ has at least $ e^{\left(\frac{1}{4}+ o_d(1) \right) \frac{\log^2 d}{d}  n }$ independent sets. 

As a simple corollary we get the following lower bound without degree restrictions.

\begin{cor}
\label{cor:count}
Let $G$ be a triangle-free graph on $n$ vertices. Then  
\[ |\mathcal I(G)| \geq e^{\left(\frac{\sqrt{2\log 2}}{4}+o(1) \right)\sqrt n \log n}. \]
\end{cor}

This improves on a result of Cooper, Dutta, and Mubayi~\cite{cooper2014hypergraphs} by a factor of $\sqrt 2$ in the exponent; they also provide a construction based on the analysis of the triangle-free process in~\cite{bohman2013dynamic,pontiveros2013triangle} showing that the optimal constant is at most $1+\log 2 \approx 1.693$ (as compared to the constant $\frac{\sqrt{2\log 2}}{4}\approx .294$ in Corollary~\ref{cor:count}).

In Section~\ref{sec:random} we show that the constants `1' in Theorem~\ref{thm:avgShort} and `1/2' in Theorem~\ref{thm:countShort} are best possible, using the random $d$-regular graph (with $d$ fixed as the number of vertices tends to infinity) as an example.  It remains an interesting open question to determine whether these constants are tight when the degree of the graph grows with the number of vertices; for example, if $d= n^c$ for some $c \in (0,1)$.

\subsection{The hard-core model}

Theorems \ref{thm:avgShort} and \ref{thm:countShort} are special cases of results on a more general model from statistical physics;  the \emph{hard-core} distribution  on the  independent sets  $\mathcal I$ of a graph $G$. For more on the hard-core model and its links with statistical physics and combinatorics, see \cite{winkler2002bethe}.  The distribution depends on a fugacity parameter $\lam > 0$ and is given by
\[
\Pr[ I] = \frac{ \lam^{\abs{I}}} { \sum_{J \in \mathcal I} \lam ^{\abs{J}}}\,.
\]
The denominator, $P_G(\lam) =\sum_{J \in \mathcal I} \lam ^{\abs{J}}$, is the \emph{partition function} of the hard-core model on $G$ (also called the \emph{independence polynomial} of $G$). If $\lam=1$, the partition function $P_G(1)$ is the total number of independent sets and the hard-core distribution is simply the uniform distribution over all independent sets of $G$.

In what follows $G$ will always be a graph on $n$ vertices. We write $\alpha(G)$ for the size of the largest independent set in $G$. The expected size of an independent set drawn from the hard-core model on $G$ at fugacity $\lam$ is
\begin{equation}
\label{eq:avgSize}
\overline\alpha_G(\lam) := \sum_{I \in \mathcal I} |I| \cdot \Pr[I] = \frac{\sum_{I \in \mathcal I} |I| \lam^{|I|}  }{ P_G(\lam) }  = \frac{\lam P_G'(\lam)}{P_G(\lam)} =\lam \cdot \left( \log P_G(\lam)  \right )'.
\end{equation}  

The key result of this paper is the following general lower bound on $\overline{\alpha}_G(\lam)$ for triangle-free graphs. 
The lower bound is written naturally in terms of the Lambert $W$ function, $W(z)$: for $z>0$, $W(z)$ denotes the unique positive real satisfying the relation $W(z)e^{W(z)} = z$. It will be useful to note that for $z\geq e$ we have $W(z)\geq\log z - \log \log z$. 

\begin{theorem}
\label{thm:Wtrianglefree}
Let $G$ be a triangle-free graph with maximum degree $d$. Then for any $\lam >0$,
\begin{equation}\label{eq:genlb}
\occ_G(\lam) \geq \frac{\lam}{1+\lam}\frac{W(d\log(1+\lam))}{d\log(1+\lam)}.
\end{equation}
\end{theorem}
The quantity $\occ_G(\lam)$, the expected fraction of vertices of $G$ in the random independent set, is known as the \emph{occupancy fraction} of $G$ at fugacity $\lam$.

It is interesting to note that the lower bound in Theorem \ref{thm:Wtrianglefree} is not monotone in $\lam$ whereas for any graph $G$, $\overline\alpha_G(\lam)$ is monotone increasing (see Proposition~\ref{prop:mono}). Simply substituting $\lam=1$ into \eqref{eq:genlb}, does not quite suffice to prove Theorem  \ref{thm:avgShort}. Surprisingly it turns out to be better to use a smaller $\lambda$ and then appeal to the monotonicity of $\overline{\alpha}_G(\lam)$. As an example, using $\lam=1/\log d$ in \eqref{eq:genlb} is enough to prove Theorem  \ref{thm:avgShort}. This shows that Theorem  \ref{thm:avgShort} holds even when we replace the average size of an independent set with a weighted average biased toward small sets. In fact we can afford to bias using any $\lam$ of the form $d^{-o(1)}$. For smaller $\lam$ the lower bound in Theorem \ref{thm:Wtrianglefree} becomes asymptotically weaker. For example taking $\lam=d^{-s}$ where $s\in (0,1)$ is fixed, \eqref{eq:genlb} gives  $\overline\alpha_G(\lam) \ge (1-s+o_d(1))\frac{\log d}{d}n$. This lower bound is tight and we in fact show that Theorem~\ref{thm:Wtrianglefree} (extended by monotonicity) is asymptotically tight for all $\lam = O_d(1)$ in Section~\ref{sec:random}.  

Turning to Theorem~\ref{thm:countShort} and the problem of counting independent sets, we again give a more general statement about the hard-core model. 
We bound the partition function of a triangle-free graph at any fugacity $\lam$ using the fact that the occupancy fraction is the scaled logarithmic derivative of the partition function.  Indeed from \eqref{eq:avgSize} it follows that
\begin{equation}\label{eq:integrateforP}
\frac{1}{n} \log P_G(\lam) = \frac{1}{n}\int_{0}^\lam \frac{\overline\alpha_G(t)}{t} \, d t.
\end{equation}

Using the lower bound on the occupancy fraction from Theorem~\ref{thm:Wtrianglefree} in the integral above gives the lower bound on the partition function which we state below, noting that Theorem~\ref{thm:countShort} is a direct consequence.

\begin{theorem}
\label{thm:numlowerbound}
Let $G$ be a triangle-free graph on $n$ vertices with maximum degree $d$. Then for all $\lam>0$,
\[ P_G(\lam) \ge \exp\left(\left[W(d\log(1+\lam))^2+2W(d\log(1+\lam))\right]\frac{n}{2d}\right).\]
In particular, taking $\lam=1$, we see that $G$ has at least $e^{\left (\frac{1}{2}+ o_d(1) \right) \frac{\log^2 d}{d}n }$ independent sets.  
\end{theorem}

\section{Lower bounds on the occupancy fraction}
\label{sec:proof}

The proof of Theorem~\ref{thm:Wtrianglefree} begins in the same way as the triangle-free case of the proof in \cite{davies2015independent} that the occupancy fraction of the hard-core model on any $d$-regular graph is at most that of $K_{d,d}$. The main difference is that in the last step we ask for the minimum instead of the maximum of the same constrained optimization problem.  The method is similar to that of Shearer~\cite{shearer1995independence} in lower bounding the independence number of  $K_r$-free graphs and that of Alon~\cite{alon1996independence} in lower bounding the independence number of graphs of average degree $d$ in which neighborhoods are $r$-colorable. Alon and Spencer~\cite{alon2011probabilistic} used this technique to give a result like Theorem~\ref{thm:avgShort} but with a worse constant.   See also chapters 21 and 22 of \cite{molloy2013graph} for an exposition of some of these results, and \cite{kostochka2014independent} for a recent application to hypergraphs. 
Unlike in previous works, in this paper we use the full power of choosing the fugacity parameter $\lam$, rather than considering only the uniform distribution.

There are two key steps:
First, we define a random variable that depends on two sources of randomness:  the random independent set drawn from the hard-core model on $G$ and a uniformly chosen random vertex $v \in V(G)$.  We then express the occupancy fraction in terms of two different expectations involving this random variable. This gives a constraint on the distribution of the random variable.
We then optimize over all random variables that satisfy the constraint, and deduce a bound on the occupancy fraction.

Recall that $\overline{\alpha}_G(\lam)$ is the expected size of an independent set drawn from the hard-core model on $G$ at fugacity $\lam$, and $\occ_G(\lam)$ is the occupancy fraction. 
In places where no confusion should arise we omit $G$ from the notation.

\begin{proof}[Proof of Theorem~\ref{thm:Wtrianglefree}]
Let $I$ be an independent set of $G$ drawn from the hard-core model at fugacity $\lam$. 
For any vertex $v$ in $G$, consider $I_0=I\setminus N(v)$ (noting that $I_0$ may include $v$ itself). We say a vertex $v$ is \emph{occupied} if $v\in I$ and \emph{unoccupied} otherwise. Furthermore we say $v$ is \emph{uncovered} if $N(v) \cap I = \emptyset$ and \emph{covered} otherwise. Note that if $v$ is covered, it must be unoccupied. We begin by stating some basic properties of the hard-core model.

\begin{description}
\item[Fact 1]
$\Pr[ v \in I | v \text{ uncovered}]= \frac{\lam}{1+\lam}$.
\end{description}
This follows from the observation that for any realization of $I_0$ under the event that $v$ is uncovered, there are exactly two possibilities for $I$: $I= I_0$ and $I= I_0 \cup\{ v\}$, since $v$ is uncovered and so may be added.

\begin{description}
\item[Fact 2]
$\Pr[ v \text{ uncovered} | v \text{ has $j$ uncovered neighbors}] =(1+\lam)^{-j}$.
\end{description}
First note that the event that $v$ is uncovered is the same as the event that its uncovered neighbors are all unoccupied. The probability that an uncovered neighbor of $v$ is unoccupied is $\frac{1}{1+\lam}$ by Fact 1. Moreover, since $G$ is triangle free the uncovered neighbors of $v$ form an independent set and so they are in fact unoccupied independently with probability $\frac{1}{1+\lam}$.

We now write the occupancy fraction as:
\begin{align}
\label{eq:alphasum1}
\occ(\lam) &=  \frac{1}{n}\sum_{v \in G} \Pr[ v \in I] \\
\label{eq:firsteq1}
&=\frac{\lam}{1+\lam}\cdot \frac{1}{n} \sum_{v \in G} \Pr[ v \text{ uncovered}] \\
\label{eq:lastline1}
&= \frac{\lam}{1+\lam}\cdot \frac{1}{n} \sum_{v \in G} \sum_{j=0}^d \Pr[ v \text{ has $j$ uncovered neighbors}] \cdot (1+\lam)^{-j}
\end{align}
where \eqref{eq:firsteq1} follows from Fact 1 and \eqref{eq:lastline1} from Fact 2.
Now let $Z$ be the random variable that counts the number of uncovered neighbors of a uniformly chosen vertex $v$. $Z$ has two layers of randomness, that of $I$ drawn from the hard-core measure, and that of selecting $v$ at random.  Interpreting the RHS of \eqref{eq:lastline1} in terms of $Z$, we obtain
\begin{equation}\label{eq:alp1bd1}
\occ(\lam) = \frac{\lam}{1+\lam} \E [ (1+\lam)^{-Z} ]\, .
\end{equation}
An alternative way to relate the occupancy fraction to $Z$ is to observe that in the sum $\sum_v \sum_{u \sim v} \Pr[u\in I]$ each vertex $u$ appears $\deg(u)$ times. 
Any uncovered neighbor of $v$ is occupied with probability $\frac{\lam}{1+\lam}$ by Fact 1 and any covered neighbor is occupied with probability zero. 
Then since $G$ has maximum degree $d$ we have
\begin{equation}\label{eq:alp2bd1}
\occ(\lam) =  \frac{1}{n}\sum_{v \in G} \Pr[ v \in I] \geq \frac{1}{dn} \sum_v \sum_{u \sim v} \Pr[u \in I] = \frac{\lam}{1+\lam}\frac{\E Z}{d}.
\end{equation}

We aim to minimize the occupancy fraction subject to the constraints on the distribution of $Z$ given by \eqref{eq:alp1bd1} and \eqref{eq:alp2bd1}. 
In fact we relax the optimization problem to optimize over all distributions of random variables $Z$ that satisfy these constraints, not only those that arise from the hard-core model on a graph. 

In the calculation below we show that  Jensen's inequality applied to \eqref{eq:alp1bd1} implies that the minimizer is achieved by the unique \emph{constant} random variable $Z$ that satisfies the constraints with equality. 
From \eqref{eq:alp1bd1} and \eqref{eq:alp2bd1} we have
\begin{align*}
\frac{1+\lam}{\lam}\cdot\occ(\lam) &\geq \max\left\{\frac{\E Z}{d} , (1+\lam)^{-\E Z} \right\} \geq \min_{x\in\mathbb{R}^+} \max\left\{\frac{x}{d} , (1+\lam)^{-x} \right\} .
\end{align*}
To compute the minimum observe that the first of our lower bounds is increasing in $x$, and the second is decreasing. 
Then the minimum occurs at the value of $x$ which makes the two bounds equal i.e. the $x$ that satisfies
\begin{align*}
xe^{\log(1+\lam) x} &= d
\shortintertext{
and hence 
}
\log(1+\lam) x &= W(d\log(1+\lam)).
\end{align*}
The result follows.
\end{proof}

A simple observation is the fact that lower bounds at a small fugacity $\lam$ imply the same bounds at higher fugacities.

\begin{prop}\label{prop:mono}
For any graph $G$, the expected size of an independent set $\overline \alpha_G(\lam)$ is monotone increasing in $\lam$.
\end{prop}
\begin{proof}
We will show that the derivative of $\overline \alpha(\lam)$ with respect to $\lam$ is positive.  We write (using $P$ for $P_G(\lam)$)
\begin{align*}
\overline \alpha'(\lam ) &= \left (\frac{\lam P'}{P} \right) ' = \frac{P'}{P} + \frac{\lam P P'' - \lam (P')^2}{P^2} \\
&= \frac{P'}{P} +\frac{1}{\lam} \left( \frac{\lam^2 P''}{P} -  \left( \frac{\lam P'}{P} \right)^2  \right ) \\
&= \frac{ \E(|I|) + \E(|I|^2) - \E(|I|) - (\E(|I|))^2   }{\lam } \\
&= \frac{\var (|I|)}{\lam} \ge0
\end{align*}
where $I$ is a random independent set drawn from the hard-core model at fugacity $\lam$. 
\end{proof}

Now using Theorem~\ref{thm:Wtrianglefree} and Proposition \ref{prop:mono} we prove the main statement of Theorem \ref{thm:avgShort}. We leave the proof of tightness to Section~\ref{sec:random}.

\begin{proof}[Proof of Theorem~\ref{thm:avgShort}]
Substituting $\lam = 1/\log d$ in \eqref{eq:genlb} and recalling the bound $W(z)\geq \log z - \log\log z$ for $z\geq e$ we obtain 
\[
\occ_G(\lam)\geq (1+o(1))\frac{\log d}{d}.
\]
The result then follows from the monotonicity given by Proposition~\ref{prop:mono}.   
\end{proof}
~
So far we have been concerned with lower bounds on the occupancy fraction of a triangle-free graph of maximum degree $d$. 
In  \cite{davies2015independent} the current authors considered upper bounds, showing that the occupancy fraction of a $d$-regular graph is at most that of $K_{d,d}$. The proof method is essentially identical to the proof of Theorem~\ref{thm:Wtrianglefree}. To make this connection explicit, we give a proof of the upper bound result here, restricted to the triangle-free case.

\begin{theorem}[\cite{davies2015independent}]\label{thm:Kahn}
Let $G$ be a $d$-regular triangle-free graph, then for any $\lam>0$
\[\occ_G(\lam)\leq \frac{1}{2d}\overline \alpha_{K_{d,d}}(\lam) \]
with equality only if $G$ is a disjoint union of copies of $K_{d,d}$.
\end{theorem}
\begin{proof}
Note that since $G$ is triangle free, equation \eqref{eq:alp1bd1} holds for $G$ and since $G$ is $d$-regular, \eqref{eq:alp2bd1} holds with equality throughout.  That is we have
\begin{equation}\label{eq:alpeq}
\occ_G(\lam) =\frac{\lam}{1+\lam}\frac{\E Z}{d}= \frac{\lam}{1+\lam} \E [ (1+\lam)^{-Z} ],
\end{equation}
where we recall that $Z$ is a random variable bounded between 0 and $d$. Now instead of asking for the minimum value of $\occ_G(\lam)$ over all distributions of $Z$ as we did in the proof of Theorem~\ref{thm:Wtrianglefree}, we ask instead for the maximum. Note that since $0\leq Z/d\leq1$ convexity of the function $x\mapsto(1+\lam)^{-x}$ implies that 
\begin{equation}\label{eq:Zbd}
(1+\lam)^{-Z}\leq\frac{Z}{d}(1+\lam)^{-d}+1-\frac{Z}{d}.
\end{equation}
Substituting this into \eqref{eq:alpeq} and using linearity of expectation yields
\begin{equation}\label{eq:EZbd}
\occ_G(\lam) =\frac{\lam}{1+\lam}\frac{\E Z}{d}\leq \frac{\lam(1+\lam)^{d-1}}{2(1+\lam)^d-1},
\end{equation}
where one can check that the right hand side is the occupancy fraction of $K_{d,d}$. 
For uniqueness, note that to have equality in \eqref{eq:EZbd} we must have had equality in \eqref{eq:Zbd} which is only possible if $Z$ takes only the values 0 and $d$. 
This distribution of $Z$ can only occur in a disjoint union of copies of $K_{d,d}$. 
To see this recall that $Z$ is the number of uncovered neighbors of a randomly selected vertex $v$.  
The only way every vertex $v$ can always have either $0$ or $d$ uncovered neighbors is for all the neighbors of $v$ to have the same neighborhood.  
For $d$-regular graphs this property holds only in graphs consisting disjoint unions of $K_{d,d}$.
\end{proof}

Note that by the integral in \eqref{eq:integrateforP}, Theorem \ref{thm:Kahn} immediately implies that for a triangle-free graph $G$, $\frac{1}{n} \log P_G(\lam) \le \frac{1}{2d} \log P_{K_{d,d}} (\lam)$. 
Taking $\lam=1$ implies that $G$ has at most as many independent sets as a disjoint union of copies of $K_{d,d}$ on the same number of vertices. 
This last statement was famously proved by Kahn \cite{kahn2001entropy} (in the case where $G$ is bipartite) using the entropy method.

\section{Counting independent sets}
\label{sec:counting} 
In this section we prove the main statement of Theorem~\ref{thm:numlowerbound} (and hence Theorem~\ref{thm:countShort}) by integrating the lower bound on the occupancy fraction of a triangle-free graph given in Theorem~\ref{thm:Wtrianglefree}. 
Again, the proof of tightness is deferred to Section~\ref{sec:random}.
We also prove Corollary~\ref{cor:count}. 

\begin{proof}[Proof of Theorem~\ref{thm:numlowerbound}]
By \eqref{eq:integrateforP} and Theorem~\ref{thm:Wtrianglefree} we have
\begin{align}
\log P_G(\lam) 
&\geq \frac{n}{d}\int_0^\lam \frac{W(d\log(1+t))}{(1+t)\log(1+t)} \; dt\\
&= \frac{n}{d}\int_0^{W(d\log(1+\lam))} (1+u) \; du\\
&= \frac{n}{2d}\left[W(d\log(1+\lam))^2+2W(d\log(1+\lam))\right].\label{eq:PGlb}
\end{align}
where for the first equality we used the substitution $u=W(d\log(1+t))$. In particular when $\lam=1$, using the inequality $W(z)\geq \log z - \log\log z$ for $z\geq e$, we have 
\[\log P_G(\lam)\ge \left(\frac{1}{2}+o_d(1)\right)\frac{\log^2 d}{d}n.\qedhere\]
\end{proof}

\begin{proof}[Proof of Corollary~\ref{cor:count}]
In a triangle-free graph the neighborhood of any vertex forms an independent set. 
Let $d$ be the largest degree of a vertex in $G$, then we have the bound 
\[
P_G(\lam)\geq\max\left\{(1+\lam)^d, \exp\left[\frac{n}{2d}W(d\log(1+\lam))^2\right]\right\},
\]
by considering the neighborhood of a vertex of maximum degree and by inequality \eqref{eq:PGlb}. 
The first expression is increasing in $d$ while the second is decreasing, and at 
\[
d=\frac{1}{2}\sqrt{\frac{n}{2\log(1+\lam)}}\log\left(\frac{n\log(1+\lam)}{2}\right)
\]
they are equal. It follows that for $\lam>0$,
\begin{equation}\label{eq:trianglepartlb}
P_G(\lam)\geq\exp\left[\frac{1}{2}\sqrt{\frac{n\log(1+\lam)}{2}}\log\left(\frac{n\log(1+\lam)}{2}\right)\right].
\end{equation}
Take $\lam=1$ to complete the proof.
\end{proof}
Inequality \eqref{eq:trianglepartlb} may be of independent interest, giving a general lower bound for the independence polynomial of a triangle-free graph on $n$ vertices.

\section{Random regular graphs}
\label{sec:random}
Here we will show that the constants in Theorems~\ref{thm:avgShort} to \ref{thm:numlowerbound} are tight, as evidenced by the random $d$-regular graph (conditioned on having no triangles).

To do this, we first leave the world of finite graphs and consider the hard-core model on the infinite $d$-regular tree $T_d$.  The hard-core model can be defined on a infinite graph by taking the limit of hard-core models on a growing sequence of finite graphs (in this case $d$-regular trees of increasing depth) with some boundary conditions specified.   
There is a unique \emph{translation invariant} hard-core measure on the infinite $d$-regular tree (see e.g. \cite{kelly1985stochastic}), and under this measure the probability that any given vertex is in the independent set is $\alpha_{T_d}(\lam)$, where   $\alpha_{T_d}(\lam)$ is the solution of the equation
\begin{align}
\label{eq:alptd}
\lam  &= \frac{\alpha}{ (1-\alpha)} \left(\frac{1-\alpha}{1-2\alpha} \right)^d.
\end{align}
Using $\alpha_{T_d}$ to denote the occupancy fraction of the infinite $d$-regular tree is standard, however we note that our notation for the analogous quantity $\occ_G(\lam)$ in a finite graph $G$ graph on $n$ vertices includes the scaling $1/n$. 
We compare the lower bound in Theorem~\ref{thm:Wtrianglefree} to $\alpha_{T_d}(\lam)$:

\begin{prop}\label{lem:Tdtight}
For every triangle-free graph $G$ of maximum degree $d$ and any $\lam = O_d(1)$, 
\[ \frac{1}{n}\overline\alpha_G(\lam) \ge (1+o_d(1))  \alpha_{T_d}(\lam). \]
\end{prop}

In the language of statistical physics, $\alpha_{T_d}(\lam)$ is the \emph{replica symmetric} occupancy fraction of a $d$-regular graph of large girth.  Extremality of a replica symmetric solution has been proved in several contexts including \cite{csikvari2014lower,ruozzi2012bethe,willsky2008loop} but usually requires a condition such as an attractive potential function (e.g. the ferromagnetic Ising model) or bipartiteness in the case of the hard-core or monomer-dimer models.   Proposition~\ref{lem:Tdtight} states that in fact under very minimal conditions (triangle free), the replica symmetric solution is a lower bound to the true occupancy fraction for all $\lam = O_d(1)$, at least asymptotically in $d$. 
This result can be compared with~\cite[Theorem 2]{davies2015independent} where we used stronger assumptions, showing that a $d$-regular, vertex-transitive bipartite graph has occupancy fraction strictly greater than $\alpha_{T_d}(\lam)$, though we only believe that result is best possible for $\lam$ below the \emph{uniqueness threshold} on the tree, $\lam \le \frac{(d-1)^{d-1}}{(d-2)^d} = \Theta\left(\frac{1}{d}\right)$, in the sense that for any $\eps>0$, there exists some finite graph $G$ with $\alpha_G(\lam) < \alpha_{T_d}(\lam)+\eps$.

\begin{proof}[Proof of Proposition~\ref{lem:Tdtight}]
In what follows all asymptotic notation refers to the limit $d\to\infty$. We first show that $\alpha_{T_d}(\lam)=(1+o(1))\frac{W(\lam d)}{d}$. 
We then show that for $\lam=o(1)$, our bound on $\overline\alpha_G(\lam)$ for triangle-free graphs given by Theorem~\ref{thm:Wtrianglefree} is asymptotically $(1+o(1))\frac{W(\lam d)}{d}$. 
We then appeal to monotonicity (Proposition \ref{prop:mono}) to extend this lower bound to all $\lam=O(1)$. 
Letting $z=\frac{\alpha_{T_d}(\lam)}{1-2\alpha_{T_d}(\lam)}d$, equation \eqref{eq:alptd} defining $\alpha_{T_d}(\lam)$ becomes
\begin{equation}\label{eq:zeq}
z\left(1+\frac{z}{d}\right)^{d-1}=\lam d.
\end{equation}
First note that from \eqref{eq:zeq} it follows easily that $z=O(\log d)$ and so $\left(1+\frac{z}{d}\right)^{d-1}=(1+o(1))e^z$. 
It then follows that $z=W(\lam d + o(\lam d))$. Using the fact that $\frac{dW}{dx}=\frac{W(x)}{x(1+W(x))}\leq \frac{W(x)}{x}$ along with the Mean Value Theorem, we deduce
\begin{equation}\label{eq:zeq2}
z=(1+o(1))W(\lam d).
\end{equation}
By the definition of $z$ we then have 
\begin{equation}\label{eq:alpTdW}
\alpha_{T_d}(\lam)= \frac{z}{d}\left(1+2\frac{z}{d}\right)^{-1}=\frac{z}{d}+O\left(\frac{z^2}{d^2}\right)=(1+o(1))\frac{W(d\lam)}{d}.
\end{equation}
Let us now suppose that $\lam= o(1)$. By Theorem~\ref{thm:Wtrianglefree} we then have
\begin{equation}
\occ_G(\lam) \geq \frac{\lam}{1+\lam}\frac{W(d\log(1+\lam))}{d\log(1+\lam)}=(1+o(1))\frac{W(d \lam+O(d \lam^2))}{d}=(1+o(1))\frac{W(d \lam)}{d},
\end{equation}
where for the last equality we use the same argument via the derivative of $W$ as above. In particular if $\lam=1/\log d$ say, we have that 
\[\occ_G(\lam) \geq (1+o(1))\frac{\log d}{d}\]
and by Proposition \ref{prop:mono}, the same holds when $\lam = \Theta(1)$. Now when $\lam = \Theta(1)$, $\alpha_{T_d}(\lam)=(1+o(1))\frac{\log d}{d}$ by \eqref{eq:alpTdW} and so the result follows.
\end{proof}
In the context of the proof technique of Section~\ref{sec:proof}, we can understand the derivation of \eqref{eq:alptd} by imposing two local constraints on the hard-core model: that every vertex has the same probability of being occupied, and that conditioned on a vertex $v$ not being occupied, the events that each of its neighbors are uncovered are \emph{independent} events.  Contrast this with our lower bound, which essentially says that the worst case for the occupancy fraction is if the number of these events that hold is always a constant; that is, the corresponding events are as negatively correlated as possible.  Now if $d$ is large, then the sum of $d$ independent indicator random variables is highly concentrated so it is plausible that the most negatively correlated case is not far from the independent case.  Proposition~\ref{lem:Tdtight} makes this precise.

Returning to the world of finite graphs, we note that in fact $\alpha_{T_d}(\lam)$ is approximately achieved by finite graphs for a very large range of $\lam$, and so Theorem~\ref{thm:Wtrianglefree} is asymptotically tight. 
Bhatnagar, Sly, and Tetali~\cite{bhatnagar2016decay} give a precise description of the local distribution of the hard-core model on a random $d$-regular graph $G_{d,n}$ for $\lam$ below a specified value $\lam_{\mathrm{upper}}  = d^{1+o(1)}$ (just below the condensation threshold of the model). 
They show that the local distribution converges to that of the unique translation invariant hard-core measure on the infinite $d$-regular tree described above, and so in particular for $\lam < \lam_{\mathrm{upper}}$,  
\begin{align*}
\occ_{G_{d,n}}(\lam) &= \alpha_{T_d}(\lam) +o_n(1), 
\end{align*}  
whp as $n \to \infty$ (note that this is a much stronger notion of approximation than that of Proposition~\ref{lem:Tdtight}, with the error term tending to $0$ with $n$ instead of $d$).  
Then provided we condition on $G_{d,n}$ being triangle free, which occurs with positive probability (depending only on $d$), a random $d$-regular graph indeed exhibits the tightness of Theorem~\ref{thm:Wtrianglefree}. Let us quickly remark that when $\lam=o_d(1)$ one can get better asymptotic agreement than that given in Proposition~\ref{lem:Tdtight}. For example, when $\lam=d^{-s}$ for some fixed $0<s<1$, one can show that both the lower bound in Theorem \ref{thm:Wtrianglefree} and $\alpha_{T_d}(\lam)$ can be written as $\frac{W(\lam d)}{d}+O_d\left(\frac{\log d}{d^{1+s}}\right)$ and so Theorem \ref{thm:Wtrianglefree} is tight in a rather strong sense in this range. 

An easy consequence of the results of \cite{bhatnagar2016decay} (either by their Sections 3.2 and 3.3 or by integrating the occupancy fraction) is that Theorem~\ref{thm:numlowerbound} is tight asymptotically in $d$ in the exponent.  In particular, whp,
\begin{equation}
 P_{G_{d,n}}(1) = e^{\left (\frac{1}{2}+ o_d(1) \right)  \frac{\log^2 d}{d} n },
 \end{equation}
showing that the constant in the exponent of Theorem~\ref{thm:countShort} is tight.

Finally let us note that for any fixed positive integer $k$, $G_{d,n}$ in fact has girth at least $k$ with positive probability (depending only on $d$). It follows that even if a stronger lower bound on girth is assumed in
Theorems~\ref{thm:avgShort} to~\ref{thm:numlowerbound}, the results cannot be improved asymptotically.

\section{On the ratio of the maximum and average independent set size}
\label{sec:conjectures}
In light of the above result showing that the average size of an independent set in a triangle-free graph with maximum degree $d$ is at least $(1+o_d(1))\frac{\log d}{d} n$, we now raise the question of whether the largest independent set should be significantly larger. 
This gives a new way to pursue an upper bound on $R(3,k)$. 
There is always a gap between the maximum and average size of an independent set (since the empty set is an independent set), but in general the ratio of maximum to average size can be arbitrarily close to 1.  For example, the complete graph $K_n$ has maximum independent set size $1$ with average size $n/(n+1)$.

We conjecture that such a narrow gap cannot occur in triangle-free graphs.
The following three conjectures make this claim precise in different ways.

\begin{conj}
\label{conj43}
For every triangle-free graph $G$, 
\[ \frac{\alpha(G)}{\overline{\alpha}_G(1)} \ge 4/3. \]
\end{conj}
Replacing $4/3$ with any number strictly greater than $1$ would give an improvement to the $R(3,k)$ bound. The graph with the smallest ratio $\alpha/\overline{\alpha}(1)$ we have found is the triangle-free cyclic graph that exhibits the bound $R(3,9)\ge36$ \cite{grinstead1982ramsey}. For this graph $\alpha/\overline{\alpha}(1)=\frac{197136}{137585} = 1.43283\dots$. We choose $4/3$ since it is a nice fraction less than $1.43$ and since it is the ratio of maximum to average size in a triangle. One might wonder if the extremal $R(3,k)$ graphs are good candidates for pushing the ratio $\alpha/\overline{\alpha}(1)$ down to $1$. However, for large $k$ it may be the case that graphs arising from the triangle-free process are asymptotically extremal, as is conjectured in \cite{pontiveros2013triangle}. We believe that for such graphs the ratio $\alpha/\overline{\alpha}(1)$ in fact converges to $2$. This motivates the following conjectures.

\begin{conj}
\label{conj:minD}
For every triangle-free graph $G$ of minimum degree $d$,
\[ \frac{\alpha(G)}{\overline{\alpha}_G(1)} \ge 2 - o_d(1). \] 
\end{conj}

\begin{conj}
\label{conj:lam}
For every $\eps >0$, there exists $\lam >0$ so that for all triangle-free graphs $G$,
\[ \frac{\alpha(G)}{\overline{\alpha}_G(\lam)} \ge 2 - \eps. \] 
\end{conj}

\begin{lemma}
The following improvements to Shearer's upper bound on $R(3,k)$ would follow from the above conjectures and Theorems~\ref{thm:avgShort} and \ref{thm:Wtrianglefree}.
\begin{enumerate}
\item Conjecture~\ref{conj43} implies $R(3,k) \le (3/4 +o(1)) k^2 /\log k$.
\item Conjecture~\ref{conj:minD} implies $R(3,k) \le (1/2 +o(1)) k^2 /\log k$.
\item Conjecture~\ref{conj:lam} implies $R(3,k) \le (1/2 +o(1)) k^2 /\log k$. 
\end{enumerate} 
\end{lemma}

\begin{proof}
(1) and (3) are immediate as Theorem~\ref{thm:Wtrianglefree} implies that for any fixed $\lam >0$, $ \overline \alpha_G(\lam) \ge (1+o_d(1))\frac{\log d}{d}  n$ in any maximum degree $d$, triangle-free graph $G$ on $n$ vertices. 

To show (2) we take any triangle-free graph of maximum degree $d$, greedily removing any vertices of degree at most $\frac{d}{2 \log d}$ and their neighbors. 
\end{proof}

One possible approach to the above conjectures is via the following simple consequence of the proof of Proposition \ref{prop:mono}. For any graph $G$
\begin{equation}
\label{eq:varIntegrate}
\alpha(G) = \lim_{\lam \to \infty} \overline\alpha_G(\lam) = \overline \alpha_G(1) + \int_{1}^{\infty} \frac{ \var_\lam (|I|)   }{\lam} \, d \lam  .
\end{equation}
In this paper we gave a lower bound for the expected size of an independent set drawn from a triangle-free graph according to the hard-core model. The above equation shows that one approach to the above conjectures would be do to the same for the variance.

It is not clear that there is anything special about excluding triangles as opposed to other graphs, and so we make the following more general conjecture.  
\begin{conj}
\label{conj:Kr}
For every $K_r$-free graph $G$, 
\[ \frac{\alpha(G)}{\overline{\alpha}_G(1)} \ge 1+ \frac{1}{r}. \]
For every $K_r$-free graph $G$ of minimum degree $d$,
\[ \frac{\alpha(G)}{\overline{\alpha}_G(1)} \ge 2 - o_d(1), \]
with $r$ fixed as $d \to \infty$.
\end{conj}

To conclude, we give a lower bound on $\alpha/\overline\alpha(\lam)$ that is tight for disjoint unions of $r$-cliques. 
The result also shows that $\alpha/\overline\alpha(1)$ is bounded away from $1$ for all graphs containing an independent set of linear size. 

\begin{theorem}
\label{thm:cliqueBound}
For any graph $G$,
\begin{align}
 \frac{\alpha(G)}{\overline \alpha_G(\lam)} &\ge 1+ \frac{\alpha(G)}{\lam n}  ,
 \intertext{or equivalently,}
 \label{part:bound1}
 P(\lam) & \ge \left(\frac{\lam }{\alpha} + \frac{1}{n}  \right) P'(\lam) .
 \end{align}
This is tight when $G$ is a union of $r$-cliques, for any $r$. 
\end{theorem}

\begin{proof}
Let $P(\lam)$ be the partition function of $G$.  Define
\[ Q(\lam) = P(\lam) - \frac{\lam}{\alpha} P'(\lam) - \frac{1}{n} P'(\lam) .\]
Then
\begin{align*}
\frac{Q(\lam)}{P(\lam)} &= 1 -  \frac{\lam}{\alpha} \frac{P'(\lam)}{P(\lam)} - \frac{1}{n} \frac{P'(\lam)}{P(\lam)} = 1- \frac{\overline\alpha(\lam)}{\alpha} - \frac{\overline\alpha(\lam)}{\lam n}, 
\end{align*}
and so after rearranging, it suffices to show that $Q(\lam) \ge 0$ for all $\lam$.  In fact, we will show that all coefficients of $Q(\lam)$ as a polynomial in $\lam$ are non-negative.  Let $Q[k]$, denote the coefficient of $\lam^k$ in $Q(\lam)$, and let $i_k = P[k]$ be the number of independent sets of size $k$ in $G$. Then
\begin{align*}
Q[k-1] &= i_{k-1}- \frac{k-1}{\alpha} i_{k-1} - \frac{k}{n} i_{k}.
\end{align*}
So it is enough to show for $k =1, \dotsc, \alpha$ that  
\begin{equation}
\label{eq:goal}
k \frac{i_{k}}{i_{k-1}} \le n - \frac{n(k-1)}{\alpha}.
\end{equation}

The following inequality is due to Moon and Moser \cite{moon1962problem} (though it is usually stated for cliques instead of independent sets):
\begin{equation}
\label{eq:moonmoser}
\frac{1}{k^2-1} \left( k^2 \frac{i_{k}(G)}{i_{k-1}(G)} - n \right ) \le \frac{i_{k+1}(G)}{i_{k}(G)} .
\end{equation}
We proceed by induction on $k$, from $k=\alpha$ down to $k=1$.  If $k=\alpha$, then \eqref{eq:moonmoser} gives:
\begin{align*}
\frac{i_{k}(G)}{i_{k-1}(G)} &\le \frac{n}{k^2} 
\end{align*}
which is exactly \eqref{eq:goal} with $\alpha = k$.  

Now inductively assume $\frac{i_{k+1}}{i_k} \le n (1- \frac{k}{\alpha}) \frac{1}{k+1}$, and plug this into the RHS of \eqref{eq:moonmoser}. Rearranging, this gives   
\begin{equation}
\label{eq:goal2}
 \frac{i_{k}}{i_{k-1}} \le n \left ( \frac{1}{k} - \frac{1}{\alpha} + \frac{1}{k \alpha} \right ) ,
\end{equation}
which is \eqref{eq:goal}. 
\end{proof}
The integrated form of \eqref{part:bound1} is a simple bound on the partition function in terms of the independence number and number of vertices: for any graph $G$ on $n$ vertices with independence number $\alpha$, 
\begin{equation*}
P_G(\lam) \le \left( 1 + \frac{\lam n}{\alpha}  \right) ^{\alpha}.
\end{equation*}
This bound is tight for unions of cliques of the same size and  is implicit in the work of Moon and Moser. It appears explicitly in~\cite{galvin2009upper}, generalizing the $\lam =1 $ case from~\cite{alekseev1991number}.

\section*{Acknowledgements}

We thank Rob Morris for his helpful comments on this paper, Noga Alon and Benny Sudakov for recalling James Shearer's 1998 SIAM talk, and James Shearer for sharing his transparencies from that talk with us.

\end{document}